\documentclass[reqno,english]{amsart}
\usepackage{amsmath,amsfonts,amssymb,graphicx,amsthm,enumerate,url}
\usepackage[noadjust]{cite}
\usepackage{stmaryrd}
\usepackage{bm}
\usepackage{dsfont}
\usepackage{nicefrac}

\usepackage{comment,paralist}
\usepackage{mathrsfs,booktabs,tabularx}
\usepackage{xifthen,xcolor,tikz,setspace,subfig}
\usepackage[export]{adjustbox}
\usetikzlibrary{decorations.pathmorphing,patterns,shapes,calc,decorations}
\usetikzlibrary{decorations.pathreplacing}
\usepackage{mathtools}
\newtheorem{theorem}{Theorem}[section]
\newtheorem{claim}[theorem]{Claim}

\newtheorem{lemma}[theorem]{Lemma}

\numberwithin{equation}{section}

\title[Noise Sensitivity of the MST of the Complete Graph]{Noise Sensitivity of the Minimum Spanning Tree of the Complete Graph}
\author{Omer Israeli}
\address{Einstein Institute of Mathematics\\
 Hebrew University\\ Jerusalem~91904\\ Israel.}
\email{omer.israeli1@mail.huji.ac.il}
\author{Yuval Peled}
\address{Einstein Institute of Mathematics\\
 Hebrew University\\ Jerusalem~91904\\ Israel.}
\email{yuval.peled@mail.huji.ac.il}

\DeclareMathOperator*{\Cov}{Cov}

\DeclareMathOperator*{\Conn}{\mathrm{conn}}

\newcommand{\bbG}{\mathbb{G}}
\newcommand{\bbF}{\mathbb{F}}
\newcommand{\bbM}{\mathbb{M}}

\newcommand{\bbE}{\mathbb{E}}

\newcommand{\bbK}{\mathbb{K}}
\newcommand{\bbI}{\mathbb{I}}
\newcommand{\cbI}{\check{\mathbb{I}}}
\newcommand{\cbM}{\check{\mathbb{M}}}
\newcommand{\bbL}{\mathbb{L}}
\newcommand{\bbP}{\mathbb{P}}
\newcommand{\Real}{\mathbb{R}}

\newcommand{\scrM}{\mathscr{M}}

\newcommand{\scrG}{\mathscr{G}}
\newcommand{\calM}{\mathcal{M}}
\newcommand{\calG}{\mathcal{G}}

\newcommand{\cS}{\mathcal{S}}
\newcommand{\cK}{\mathcal{K}^{\infty}}
\newcommand{\cKj}{\mathcal{K}^{\infty}_{\mathrm{joint}}}

\newcommand{\bbJ}{\mathbb{J}}
\newcommand{\GHP}{\textsc{ghp}}
\newcommand{\ard}{\xrightarrow{\,\mathrm{d}\,}}
\newcommand{\arp}{\xrightarrow{\,\mathrm{p}\,}}
\newcommand{\bbGn}{\bbG_{n,\lambda}}
\newcommand{\bbGne}{\bbG^\varepsilon_{n,\lambda}}
\newcommand{\bbFn}{\bbF_{n,\lambda}}
\newcommand{\bbFne}{\bbF^\varepsilon_{n,\lambda}}

\newcommand{\bbMn}{\bbM_{n,\lambda}}
\newcommand{\bbMne}{\bbM^\varepsilon_{n,\lambda}}

\begin{document}
\maketitle
\begin{abstract}
We study the noise sensitivity of the minimum spanning tree (MST) of the $n$-vertex complete graph when edges are assigned independent random weights. 
It is known that when the graph distance is rescaled by $n^{1/3}$ and vertices are given a uniform measure, the MST converges in distribution in the Gromov-Hausdorff-Prokhorov (GHP) topology~\cite{AddarioBerry2013}.
We prove that if the weight of each edge is resampled independently with probability $\varepsilon\gg n^{-1/3}$, then the pair of rescaled minimum spanning trees --- before and after the noise --- converges in distribution to independent random spaces. Conversely, if $\varepsilon\ll n^{-1/3}$, the GHP distance between the rescaled trees goes to $0$ in probability. This implies the noise sensitivity and stability for every property of the MST that corresponds to a continuity set of the random limit. The noise threshold of $n^{-1/3}$ coincides with the critical window of the Erd\H{o}s-R\'enyi random graphs. In fact, these results follow from an analog theorem we prove regarding the minimum spanning forest of critical random graphs.
\end{abstract}

\section{Introduction}
\label{Section: Introduction}
The minimum spanning tree (MST) of a weighted graph is a classical object in discrete mathematics, whose research goes back to Bor\r{u}vka's Algorithm from 1926 (see ~\cite{Boruvka}).
Denote by $\bbM_n$  the MST of the $n$-vertex complete graph $K_n$ assigned with independent $\mathrm{U}[0,1]$-distributed edge weights $W_n=(w_e)_{e\in K_n}$.
Frieze~\cite{Frieze1985} famously showed that the expected total weight of $\bbM_n$ converges to $\zeta(3)$, initiating an extensive study of the distribution of the total weight (e.g.,~\cite{Janson1995,Janson2006}). From a purely graph-theoretic perspective, a decade old fundamental work on the metric structure of $\bbM_n$ by Addario-Berry, Broutin, Goldschmidt, and Miermont~\cite{AddarioBerry2013}, which plays a key role in this paper, discovered the existence of a scaling limit of $\bbM_n$ as a measured metric space. An explicit construction of the limit was recently obtained in ~\cite{broutin2023convex}. In addition, the local weak limit of $\bbM_n$ was studied in ~\cite{Addarioberry2013local,angel2023scaling}.

The notion of noise sensitivity of Boolean functions, that was introduced by Benjamini, Kalai, and Schramm in ~\cite{Benjamini1998}, can be directly applied to the random MST. Namely, let $\varepsilon=\varepsilon_n$ be a noise parameter, and $W_n^\varepsilon=(w^{\varepsilon}_e)_{e\in K_n}$ be obtained from $W_n$ by resampling each $w_e$ independently with probability $\varepsilon$. The MST of $K_n$ with respect to the new weights $W_n^\varepsilon$ is denoted by $\bbM_n^{\varepsilon}$. Suppose $f_n$ is a sequence of Boolean functions defined on $n$-vertex trees, such that $\bbE[f_n(\bbM_n)]$ is bounded away from $0$ and $1$ as $n\to\infty$. We say that the sequence $f_n$ is \emph{ $\varepsilon$-noise sensitive} (resp. \emph{stable}) if $\Cov(f_n(\bbM_n),f_n(\bbM_n^{\varepsilon}))\to 0$ (resp. $1$) as $n\to\infty$. This paper deals with the noise sensitivity and stability of (functions that depend on) the scaled measured metric structure of $\bbM_n$.

\subsection{The metric structure of the random MST} 
The tree $\bbM_n$ is closely related to the Erd\H{o}s-R\'enyi random graph. Kruskal's algorithm~\cite{Kruskal1956} computes the tree $\bbM_n$ by starting from an empty $n$-vertex graph and adding edges according to their (uniformly random) increasing weight order, unless the addition of an edge forms a cycle. Therefore, the minimum spanning forest (MSF) $\bbM(n,p)$ of the random graph $\mathbb G(n,p):=\{e\in K_n:w_e\le p\}$ (endowed with the random weights from $W_n$) is a subgraph of $\bbM_n$. Indeed, $\bbM(n,p)$ is one of the forests en route $\bbM_n$ in Kruskal's algorithm. In addition, $\bbM(n,p)$ can be obtained from $\bbG(n,p)$ using a cycle-breaking algorithm, i.e., by repeatedly deleting the heaviest edge participating in a cycle until the graph becomes acyclic (see \S\ref{Section: Preliminaries}). 

Fix $\lambda \in \Real$ and let $p(n,\lambda)=1/n + \lambda/n^{4/3}$. We denote the critical random graph $\bbGn:=\mathbb G(n,p(n,\lambda))$ and its MSF $\bbMn:=\mathbb M(n,p(n,\lambda))$. These graphs play a key role in the study of the MST. It is shown in ~\cite{AddarioBerry2013}  (in a sense we precisely specify below), that for a large constant $\lambda$, ''most" of the global metric structure of $\bbM_n$ is present in its subgraph $\bbMn$.
The size and structure of the connected components of $\bbGn$ have been studied extensively ~\cite{Luczak1994}. In his work on multiplicative coalescence, Aldous~\cite{Aldous1995} determined the limit law of the random sequence of the sizes of the connected components of $\bbGn$, given in decreasing order and rescaled by $n^{-2/3}$. The limit law is beautifully expressed via a reflected Brownian motion with a parabolic drift. A breakthrough result of Addario-Berry, Broutin and Goldschmidt~\cite{AddarioBerry2009} discovered the scaling limit in Gromov--Hausdorff distance of the connected components of $\bbGn$ viewed as metric spaces. 

In \cite{AddarioBerry2013}, these authors and Miermont extended this result to {\em measured} metric spaces in the Gromov--Hausdorff--Prokhorov (GHP) distance. In addition, by applying a continuous cycle-breaking algorithm on the scaling limit of the components, they discovered the scaling limit of $\bbM_n$. More formally, let $\calM$ be the space of isometry-equivalence classes of compact measured metric spaces endowed with the GHP distance. Denote by $M_n\in\calM$ the measured metric space obtained from $\bbM_n$ by rescaling graph distances by $n^{-1/3}$ and assigning a uniform measure on the vertices. The main theorem in \cite{AddarioBerry2013}  asserts that there exists a random compact measured metric space $\scrM$ such that $M_n\ard\scrM$ in the space $(\calM,d_{\GHP})$ as $n\to\infty$. The limit $\scrM$ is an $\Real$-tree that, remarkably, differs from the well-studied CRT~\cite{LeGall2005}. 

\subsection{Noise sensitivity and stability}
Noise sensitivity of Boolean functions captures whether resampling only a small, $\varepsilon$-fraction, of the input bits of a function leads to an almost independent output.
Since its introduction in ~\cite{Benjamini1998}, this concept has found various applications in theoretical computer science~\cite{mossel2005noise} and probability theory~\cite{garban_steif_2014}. 
Lubetzky and Steif ~\cite{Lubetzky2015} initiated the study of the noise-sensitivity of critical random graphs. Denote by $\bbGne$ the graph that is obtained by independently resampling each edge according to its original $
\mathrm{Ber}(p(n,\lambda))$ distribution with probability $\varepsilon$. They proved that the property that the graph contains a cycle of length in $(an^{1/3},bn^{1/3})$ is noise sensitive provided that $\varepsilon \gg n^{-1/3}$. Heuristically, a threshold of $n^{-1/3}$ for noise-sensitivity of such  ``global" graph properties seems plausible. Indeed, if $\varepsilon\gg n^{-1/3}$, then the edges that are not resampled, and appear in the graph both before and after the noise operation, form a subcritical random graph in which the property in question is degenerate. 

Roberts and \c{S}eng\"{u}l~\cite{Roberts2018} established the noise sensitivity of properties related to the size of the largest component of $\bbGn$, under the stronger assumption that $\varepsilon\gg n^{-1/6}$. Afterwards, the above heuristic was made rigor in ~\cite{Lubetzky2020} by Lubetzky and the second author, establishing that if $\varepsilon\gg n^{-1/3}$ both (i) the rescaled sizes and (ii) the rescaled measured metric spaces, obtained from the components of $\bbGn$ and $\bbGne$, are  asymptotically independent (where the entire sensitivity regime was completed in ~\cite{Frilet2021}). On the other hand, if $\varepsilon\ll n^{-1/3}$ the effect of the noise was shown to be negligible. Rossignol identified non-trivial correlations when $\varepsilon=tn^{-1/3}$~\cite{Rossignol2021}.

In the same manner the measured metric space $M_n\in\calM$ is obtained from $\bbM_n$, let $M_n^\varepsilon\in\calM$ denote the measured metric space obtained from $\bbM_n^\varepsilon$ by rescaling the graph distances by $n^{-1/3}$ and assigning a uniform measure on the vertices. Our main theorem establishes a noise threshold of $n^{-1/3}$ for any sequence of functions that depend on the scaled measured metric space. This threshold coincides with the noise threshold for critical random graphs and, accordingly, with the width of the critical window in the Erd\H{o}s-R\'enyi phase transition.

\begin{theorem}
\label{theorem: maintheorem}
Let $\varepsilon=\varepsilon_n>0$. Then, as $n\to\infty$,
\begin{enumerate}
    \item
    \label{part: 1}
    If $\varepsilon^{3}n\to \infty$ then the pair $\left(M_n,M_{n}^{\varepsilon}\right)$ converges in distribution to a pair of independent copies of $\scrM$ in $(\calM,d_{\GHP})$.
    \item
    \label{part: 2}
    If $\varepsilon^{3}n\to0$ then 
        $
            d_{\GHP}(M_{n},M_{n}^{\varepsilon})\overset{p}{\to}0
        $.
\end{enumerate}
\end{theorem}

For any sequence $f_n(\bbM_n):=\mathbf{1}_{M_n\in S}$ of Boolean functions, where $S$ is a continuity set of the limit space $\scrM$, our theorem implies $\varepsilon$-noise sensitivity if $\varepsilon\gg n^{-1/3}$ in Part \eqref{part: 1}, and $\varepsilon$-noise stability if $\varepsilon\ll n^{-1/3}$  in Part \eqref{part: 2}. For concrete examples, indicator functions of properties such as ``the diameter of the tree is at most $b\cdot n^{1/3}$," or ``the average distance between a pair of vertices is greater than $a\cdot n^{1/3}$" naturally arise. However, we leave the verification that these examples indeed correspond to continuity sets of $\scrM$ for future work (see Section \ref{sec:open}), noting that it appears to follow from the recent explicit construction of $\scrM$ as the Brownian parabolic tree ~\cite{broutin2023convex}.

\subsection{The random minimum spanning forest}
Following ~\cite{AddarioBerry2013}, our approach for Theorem \ref{theorem: maintheorem} starts by investigating the effect of the noise operator on the metric structure of $\bbMn$. The forest $\bbMne$ denotes the MSF of the graph $\bbGne := \{e\in K_n~:~w_e^{\varepsilon}\le p(n,\lambda)\}$ endowed with weights from $W_n^\varepsilon.$

For an $n$-vertex graph $G$ and an integer $j\ge 1$, let $\cS_j(G)$ be obtained from the $j$-th largest connected component of $G$ by rescaling the graph distances by $n^{-1/3}$ and assigning each vertex a measure of $n^{-2/3}.$ We denote by $\cS(G)$ the sequence $\cS(G)=(\cS_j(G))_{j\ge 1}$ of elements in $\calM$. We consider the two sequences of scaled measured metric spaces, given by $M_{n,\lambda}:=\cS(\bbMn)$ and  $M_{n,\lambda}^\varepsilon:=\cS(\bbMne)$. For every two sequences $S,S'$ of elements in $\calM$, let $d_\GHP^4(S,S') = (\sum_j d_\GHP(S_j,S_j')^4)^{\frac14}$ and set $\mathbb{L}_4=\{ S\in\calM^{\mathbb N}: \sum_j d_\GHP(S_j,\mathsf{Z})^4<\infty\}$ where $\mathsf{Z}$ is the zero metric space.  

It is shown in ~\cite{AddarioBerry2013} that there exists a sequence $\scrM_{\lambda}:=(\scrM_{\lambda,j})_{j\ge 1}$ of random compact measured metric spaces such that $M_{n,\lambda} \to \scrM_{\lambda}$ as $n\to\infty$ in distribution in $(\mathbb{L}_4,d_\GHP^4).$
The connection between $\bbM_n$ and $\bbMn$ from \cite[Theorem 1.2]{AddarioBerry2013} that was mentioned above can be now stated precisely. That is, if we let $\hat\scrM_{\lambda,1}$ be obtained from $\scrM_{\lambda,1}$ by renormalizing its measure to a probability measure, then $\hat\scrM_{\lambda,1}\ard \scrM$ in $d_\GHP$ as $\lambda\to\infty$. Hence, Theorem \ref{theorem: maintheorem} is derived from the following theorem.

\begin{theorem}
\label{theorem: secondtheorem}
Let $\lambda\in\Real$ and $\varepsilon=\varepsilon_n>0$.
\begin{enumerate}
    \item
    \label{part: GNP1}
    If $\varepsilon^{3}n\to \infty$ as $n\to\infty$, then the pair $\left(M_{n,\lambda},M_{n,\lambda}^\varepsilon\right)$ converges in distribution to a pair of independent copies of $\scrM_{\lambda}$ in $(\mathbb{L}_4,d_{\GHP}^4)$.
    \item
    \label{part: GNP2}
    If $\varepsilon^{3}n\to0$ as $n\to\infty$, then 
        $
            d_{\GHP}^4(M_{n,\lambda},M_{n,\lambda}^{\varepsilon})\overset{p}{\to}0.$
\end{enumerate}
\end{theorem}

The noise sensitivity of critical random graphs from ~\cite{Lubetzky2020} and \cite{Frilet2021} establishes that if $\varepsilon^{3}n\to \infty$ then the scaled measure metric spaces of the components of $\bbGn$ and $\bbGne$ are asymptotically independent. This fact seemingly excludes any non-negligible correlation between the scaled measure metric spaces of $\bbMn$ and $\bbMne$, which are obtained from $\bbGn$ and $\bbGne$ respectively by the cycle-breaking algorithm. However, the existence of ``bad" edges that participate in cycles in both graphs,  and with the same (not resampled) weight, may correlate the two runs of the cycle-breaking algorithm. We analyze the joint cycle-breaking algorithm and prove that if $\varepsilon^{3}n\to \infty$ then, with high probability, the number of such ``bad" edges is too small to generate a non-negligible correlation. For the stability part, we show that if $\varepsilon^{3}n\to 0$ then, typically, the two runs of the cycle-breaking algorithm are  identical.

The remainder of the paper is organized as follows. Section \ref{Section: Preliminaries} contains some preliminaries and additional background material needed for the proof of the main results. In Section \ref{sec:MSF} we prove both parts of Theorem \ref{theorem: secondtheorem}, and in Section \ref{sec:MST} we complete the proof of Theorem \ref{theorem: maintheorem}. We conclude with some open problems in Section \ref{sec:open}.

\section{Preliminaries}
\label{Section: Preliminaries}
\subsection{Notations}
\label{subsec: notations}
For clarity, we briefly recall the notations that were interspersed within the introduction and present some additional concepts needed in the proofs.
Let $n$ be an integer and $K_n$ the complete $n$-vertex graph. The edges of $K_n$ are assigned independent and $\mathrm{U}[0,1]$-distributed weights $W_n:=(w_e)_{e\in K_n}$. 
Given a noise-parameter $\varepsilon=\varepsilon_n$, we define the weights $W_n^\varepsilon:=(w^{\varepsilon}_e)_ {e\in K_n}$ by
\[
     w^\varepsilon_e := 
     \begin{cases}
      w_e & b_e = 0 \\
      w'_e & b_e =1 \\
     \end{cases},
\]
where $b_e$ is an independent $\mathrm{Ber}(\varepsilon)$ random variable and $w'_e$ is an independent $\mathrm{U}[0,1]$-distributed weight. In words, we independently, with probability $\varepsilon$, {\em resample} the weight of each edge. 

All the random graphs we study are measurable with respect to $W_n,W_n^\varepsilon$. Namely, $\bbM_n,\bbM_n^\varepsilon$ are the minimum spanning trees (MST) of $K_n$ under the weights $W_n,W_n^\varepsilon$ respectively. In addition, we always refer to $p$ as $p:=p(n,\lambda)=1/n+\lambda/n^{4/3}$, where $\lambda\in\Real$, and denote the random graphs
\[
\bbGn := \{e\in K_n~:~w_e\le p\}\mbox{, and }\bbGne := \{e\in K_n~:~w_e^{\varepsilon}\le p\}\,.
\]
Note that as random (unweighted) graphs, $\bbGne$ is obtained from $\bbGn$ by applying the standard $\varepsilon$-noise-operator that independently, with probability $\varepsilon$, resamples each edge.
We denote the intersection of these two graphs by $\bbI:=\bbGn\cap\bbGne$, and 
its subgraph
$$
\cbI = \{ e\in \bbGn\cap\bbGne~:~b_e=0\},
$$
consisting of the edges that appear in $\bbGn$ and whose weight was not resampled --- and thus also appear in $\bbGne$. We denote by $\bbMn$ (resp. $\bbMne$)  the minimum spanning forest (MSF) of $\bbGn$ (resp. $\bbGne$) when endowed with edge weights from $W_n$ (resp. $W_n^\varepsilon$).

To some of the random graphs above, we associate a scaled measured metric space in $\calM$. Recall that $\cS(G)$ is a sequence of elements in $\calM$ that is obtained from an $n$-vertex graph $G$ by ordering its components in decreasing size, rescaling the graph distances by $n^{-1/3}$ and assigning each vertex a measure of $n^{-2/3}.$ We denote 
$M_{n,\lambda}=\cS(\bbMn),~M_{n,\lambda}^\varepsilon=\cS(\bbMne),~
G_{n,\lambda}=\cS(\bbGn)$ and $G_{n,\lambda}^\varepsilon=\cS(\bbGne).$
We sometime refer to specific elements in these sequences, e.g., 
$M_{n,\lambda,j}$ denotes the measured metric space obtained from the $j$-th largest component $C_j(\bbGn)$ of the graph $\bbGn$. In addition, given a connected graph $G$, let $\hat{\cS}(G)$ be obtained from $G$ by rescaling the graph distance by $n^{-1/3}$ and assigning a uniform probability measure on its vertices. We view $M_n = \hat{\cS}(\bbM_n)$, $M_n^\varepsilon = \hat{\cS}(\bbM_n^\varepsilon)$ as elements of $\calM$.

 \subsection{The Joint Cycle Breaking Algorithm}
\label{Subsection: The Joint Cycle Breaking Algorithm}

An alternative approach to the well-known Kruskal's algorithm for finding the MSF of a weighted graph is the cycle-breaking algorithm, aka the reverse-delete algorithm, which was also introduced by Kruskal in \cite{Kruskal1956}. Consider $\Conn(G)$, the set of edges of $G$ that participate in a cycle. In other words, $e\in\Conn(G)$ if removing it does not increase the number of connected components. The algorithm finds the MSF of a given weighted graph $G$ by sequentially removing the edge with the largest weight from $\Conn(G)$. Once the remaining graph is acyclic, its edges form the MSF of $G$.

For a graph $G$, let $\cK(G)$ denote the random MSF of $G$ if the edges are given exchangeable, distinct random weights. In such a case, $\cK(G)$ can be sampled by running a cycle-breaking algorithm on $G$ that removes a uniformly random edge from $\Conn(G)$ in each step. Indeed, the heaviest edge in $\Conn(G)$ is uniformly distributed, regardless of which edges were exposed as the heaviest in the previous steps of the algorithm. For example, conditioned on (the edge set of) $\bbGn$, the forest $\bbMn$ is $\cK(\bbGn)$-distributed. 

Given two finite graphs $G_1,G_2$ and a common subgraph $H\subset G_1\cap G_2$, let $W^i:=(w^i_e)_{e\in G_i},~i=1,2,$ be two exchangeable random weights given to the edges of $G_1$ and $G_2$ that are distinct except that
$
w^1_e=w^2_e \iff e\in H.
$ We denote by $\cKj(G_1,G_2,H)$ the joint distribution of the pair of minimum spanning forests of $G_1,G_2$ under the above random edge weights $W^1,W^2$. 

Clearly, the marginal distributions of $\cKj(G_1,G_2,H)$ are $\cK(G_1)$ and $\cK(G_2)$. In addition, if $H\cap\Conn(G_1)\cap\Conn(G_2)=\emptyset$  then $\cKj(G_1,G_2,H)=\cK(G_1)\times\cK(G_2)$, i.e., the joint cycle-breaking algorithm can be carried out by two independent cycle-breaking algorithms on $G_1$ and $G_2$. 
On the other extreme, if $\Conn(G_1)=\Conn(G_2)$ and $\Conn(G_1)\subseteq H$, then the exact same set of edges is removed in both  graphs during the run of the joint cycle breaking algorithm. In such a case, if $(M_1,M_2)\sim\cKj(G_1,G_2,H)$ then $M_1\sim\cK(G_1)$ and $M_2$ is then deterministically defined by $M_2=G_2\setminus (G_1\setminus M_1)$. 

The example prompting this definition in our study is that, conditioned on (the edge sets of) $\bbGn,\bbGne,\cbI$ defined in \S\ref{subsec: notations}, the distribution of the pair $(\bbMn,\bbMne)$ is $\cKj(\bbGn,\bbGne,\cbI)$.
Indeed, among the edges in $\bbGn\cup\bbGne$, only those in $\cbI$ have the same weight in $W_n$ and $W_n^\varepsilon$, and all the other weights are independent. Roughly speaking, the two extreme cases for $H$ mentioned above describe what typically occurs in the noise sensitivity and stability regimes.

\subsection{Scaling limits}
We conclude this section by briefly reviewing previous works regarding the scaling limits of the measured metric spaces obtained from the random graphs that appear in our work. In ~\cite{AddarioBerry2013} (building on results from ~\cite{AddarioBerry2009}) it is proved that there exists a sequence $\scrG_\lambda=(\scrG_{\lambda,j})_{j\ge 1}$ of random elements in $\calM$ such that $G_{n,\lambda}\ard \scrG_\lambda$ in $(\mathbb L_4,d_{\GHP}^4)$ as $n\to\infty$. Furthermore, by defining a continuous version  of the cycle-breaking algorithm (whose distribution is also denoted by $\cK$), they obtain a sequence $\scrM_\lambda=(\scrM_{\lambda,j})_{j\ge 1}$ of random elements in $\calM$ which is $\cK(\scrG_\lambda)$-distributed conditioned on $\scrG_\lambda$. 
They prove that $M_{n,\lambda}\ard \scrM_\lambda$ in $(\mathbb L_4,d_{\GHP}^4)$ as $n\to\infty$ by establishing the continuity of $\cK$, and that the scaling limit $\scrM$ of $M_n$ is obtained by renormalizing the measure of $\scrM_{\lambda,1}$ to a probability measure and taking $\lambda\to\infty$ (as mentioned in \S \ref{Section: Introduction}).

\section{Proof of Theorem \ref{theorem: secondtheorem}}\label{sec:MSF}
\subsection{Noise Sensitivity of the MSF}
\label{Subsection: Noise Sensitivity}
We saw that the pair $(\bbMn,\bbMne)$ is obtained by a joint cycle-breaking algorithm and that it is $\cKj(\bbGn,\bbGne,\cbI)$ - distributed.
Our first goal is to show that, if $\varepsilon^3n\to\infty$, the joint cycle-breaking is close to two {\em independent} runs of the cycle-breaking algorithm.
We start by bounding the number of edges that participate in a cycle in both graphs, and, as a result, can potentially correlate the two forests during the joint cycle-breaking. 
\begin{lemma}
\label{lem:com_core}
    Fix $\lambda\in\Real$ and let $\varepsilon^3 n\to\infty$, $\bbGn,\bbGne$ as defined in \S\ref{Section: Preliminaries}, and $\bbJ= \Conn(\bbGn)\cap\Conn(\bbGne).$ Then, 
    
$$\bbP(|\bbJ|>\omega\varepsilon^{-1}) \to 0\,,$$
as $n\to\infty$ for every diverging sequence $\omega=\omega(n)\to\infty.$
\end{lemma}
In the proof below, we denote by $G-e$ the subgraph of $G$ on the same vertex set with the edge set $E(G)\setminus\{e\}$, and by $G\setminus A$ the subgraph of $G$ induced by the vertices that are not in the vertex subset $A$. 

\begin{proof}
Recall that  $\bbI$ denotes the intersection $\bbGn \cap \bbGne$. The graph $\bbI$ is a $\mathcal G(n,\theta)$ random graph, where $$\theta:=p(1-\varepsilon + \varepsilon p)=\frac{1-\varepsilon(1+o(1))}{n}.$$ Fix some edge $e=\{u,v\}$ in $K_n$. We consider two  disjoint possibilities for the occurrence of  the event $e\in\bbJ$:
\begin{enumerate}
    \item \label{enum3.1: 1} The event $A=A_e=\{e\in \Conn(\bbI)\}$ where $e$ belongs to a cycle that is contained in both graphs, or 
    \item the event $B=B_e=\{e\in \bbJ\setminus\Conn(\bbI)\}$ where there are two distinct cycles in $\bbGn$ and $\bbGne$ both containing $e$, and there is no cycle in $\bbI$ containing $e.$
\end{enumerate}

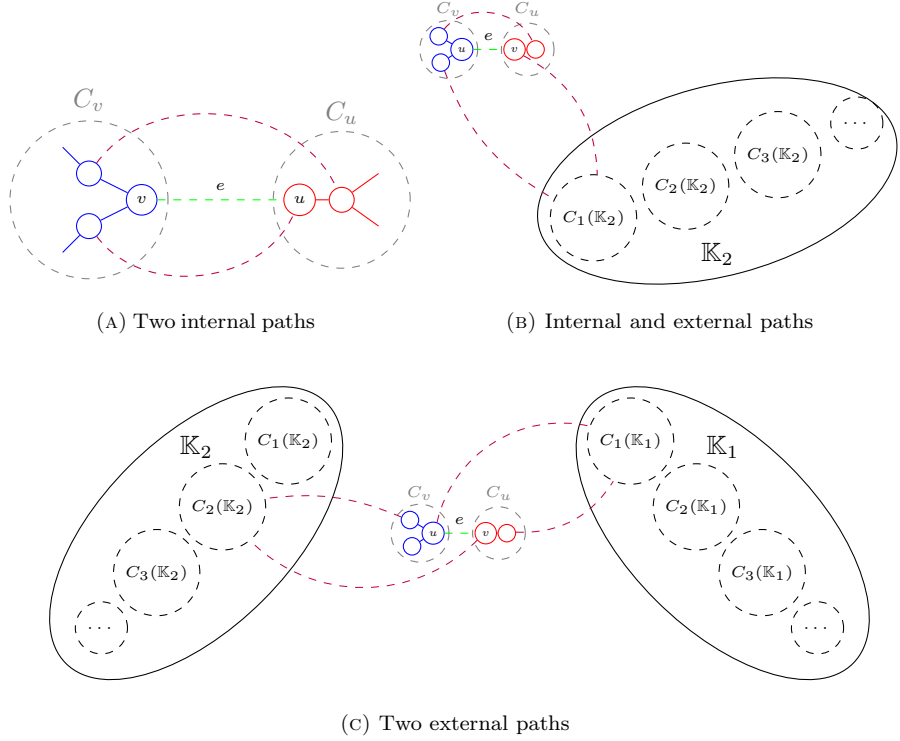
\begin{figure}[ht]
\centering
\subfloat[Two internal paths]{
\begin{tikzpicture}[scale= 0.7]
\label{fig: 1B}
        \node[shape=circle, draw=red, node font = \tiny ] (j) at (1.5, 2) {$u$};
        \node[shape=circle, draw=blue, node font = \tiny] (i) at (-1.5,2) {$v$};
        \node[shape=circle, draw=blue] (A) at (-2.5,1.5) {};
        \node[shape=circle, draw=blue] (B) at (-2.5,2.5) {};
        \node[shape=circle, draw=red] (C) at (2.3,2) {};
        \path[green, dashed, font = \tiny] (i) edge node[above] {\textcolor{black}{$e$}} (j);
        
        \path[blue] (i) edge node[above] {} (A);
        \path[blue] (i) edge node[above] {} (B);
        \draw[blue] (A) -- (-3,1);
        \draw[blue] (B) -- (-3,3);
        \draw[gray,dashed] (-2.5, 2) circle (1.5) node[above=30] {$C_v$};
        \path[red] (j) edge node[above] {} (C);
        \draw[red] (C) -- (3,2.5);
        \draw[red] (C) -- (3,1.5);
        \draw[gray ,dashed] (2.3, 2) circle (1.3) node[above=25] {$C_u$};
        \path[purple, dashed, bend left = 60,thin, font = \tiny] (B) edge node[above] {} (C);
        \path[purple, dashed, bend right = 60,thin, font = \tiny] (A) edge node[above] {} (j);
    \end{tikzpicture}}
\subfloat[Internal and external paths]{
    \begin{tikzpicture}[scale = 0.7]
    \label{fig: 1C}
        \node[shape=circle, draw=red, node font = \tiny, scale = 0.7] (j) at (-1.5, -2) {$v$};
        \node[shape=circle, draw=blue,node font = \tiny, scale = 0.7] (i) at (-2.5,-2) {$u$};
        \node[shape=circle, draw=blue, scale = 0.7] (A) at (-2.9,-2.25) {};
        \node[shape=circle, draw=blue, scale = 0.7] (B) at (-2.94,-1.75) {};
        \node[shape=circle, draw=red, scale = 0.7] (C) at (-1.1,-2) {};
        \path[green, dashed, font = \tiny,scale = 0.7] (i) edge node[above] {\textcolor{black}{$e$}} (j);
        \path[blue] (i) edge node[above] {} (A);
        \path[blue] (i) edge node[above] {} (B);
         \path[red] (j) edge node[above] {} (C);
        \draw[gray,dashed] (-2.75, -2 ) circle (0.55) node[above =10 , node font = \tiny] {$C_v$};
        \draw[gray ,dashed] (-1.25, -2) circle (0.5) node[above = 10,  node font = \tiny] {$C_u$};
        
        \node[shape=circle, draw=black, dashed,font  = \tiny] (CC1) at (0,-5.2) {$C_1(\bbK_2)$}; 
        \node[shape=circle, draw=black, dashed,font  = \tiny] (CC2) at (1.75,-4.6) {$C_2(\bbK_2)$}; 
        \node[shape=circle, draw=black, dashed,font = \tiny] (CC3) at (3.5,-4) {$C_3(\bbK_2)$};
        \node[shape=circle, draw=black, dashed,font = \tiny] (CC4) at (5,-3.4) {$\cdots$};
        \draw[rotate around ={-75:(2.35,-4.5)}] (2.35,-4.5) ellipse (1.8cm and 3.5cm) node[below = 20] {$\bbK_2$};

        \path[purple, dashed, bend left = 60] (B) edge node[] {} (C);
        \path[purple, dashed, bend right = 20] (A) edge node[] {} (CC1);
        \path[purple, dashed, bend left = 30] (j) edge node[] {} (CC1);
    \end{tikzpicture}}    
    \\
    \subfloat[Two external paths]{
    \begin{tikzpicture}[scale = 0.7]
    \label{fig: 1D}
        \node[shape=circle, draw=red, node font = \tiny, scale = 0.7] (j) at (0.5, 0) {$v$};
        \node[shape=circle, draw=blue,node font = \tiny, scale = 0.7] (i) at (-0.5,0) {$u$};
        \path[green, dashed, font = \tiny,scale = 0.7] (i) edge node[above] {\textcolor{black}{$e$}} (j);
        
        \node[shape=circle, draw=blue, scale = 0.7] (A) at (-0.9,-0.25) {};
        \node[shape=circle, draw=blue, scale = 0.7] (B) at (-0.94,0.25) {};
        \draw[gray,dashed] (-0.75, 0 ) circle (0.55) node[above =10 , node font = \tiny] {$C_v$};
        
        \node[shape=circle, draw=red, scale = 0.7] (C) at (0.9,0) {};
        \draw[gray ,dashed] (0.75, 0) circle (0.5) node[above = 10,  node font = \tiny] {$C_u$};
        
        \path[blue] (i) edge node[above] {} (A);
        \path[blue] (i) edge node[above] {} (B);
         \path[red] (j) edge node[above] {} (C);

        \node[shape=circle, draw=black, dashed,font  = \tiny] (C1) at (3.25,1.75) {$C_1(\bbK_1)$}; 
        \node[shape=circle, draw=black, dashed,font  = \tiny] (C2) at (4.5,0.5) {$C_2(\bbK_1)$}; 
        \node[shape=circle, draw=black, dashed,font = \tiny] (C3) at (5.75,-0.75) {$C_3(\bbK_1)$};
        \node[shape=circle, draw=black, dashed,font = \tiny] (C4) at (6.8,-1.8) {$\cdots$};
        \draw[rotate around ={45:(5,0)}] (5,0) ellipse (1.8cm and 3.5cm) node[above = 25] {$\bbK_1$};
        
        \node[shape=circle, draw=black, dashed,font  = \tiny] (CC1) at (-3.25,1.75) {$C_1(\bbK_2)$}; 
        \node[shape=circle, draw=black, dashed,font  = \tiny] (CC2) at (-4.5,0.5) {$C_2(\bbK_2)$}; 
        \node[shape=circle, draw=black, dashed,font = \tiny] (CC3) at (-5.75,-0.75) {$C_3(\bbK_2)$};
        \node[shape=circle, draw=black, dashed,font = \tiny] (CC4) at (-6.8,-1.8) {$\cdots$};
        \draw[rotate around ={-45:(-5,0)}] (-5,0) ellipse (1.8cm and 3.5cm) node[above = 25] {$\bbK_2$};
        \path[purple, dashed, bend right = 15, very thin] (B) edge node[] {} (CC2);
        \path[purple, dashed, bend left = 45, very thin] (j) edge node[] {} (CC2);
        
        \path[purple, dashed, bend left = 45, very thin] (i) edge node[] {} (C1);
        \path[purple, dashed, bend right = 30,very thin] (C) edge node[] {} (C1);
         
    \end{tikzpicture}}
    \caption{The three combinations of internal and external paths between $u$ and $v$ that can cause the occurrence of $B$.}
    \label{fig: exposure of Gn and Gne}
\end{figure}
We bound the probability of $A$ by observing it occurs if and only if $e\in\bbI$ and there is a path in the graph $\bbI - e$ from $v$ to $u$.
By enumerating all the paths from $v$ to $u$ with $k\ge 1$ additional vertices we find that
\begin{equation}\label{eq:A}
    \bbP(A) \le \theta\sum_{k\ge 1}n^k\theta^{k+1} \le \frac{\theta^3 n}{1-\theta n}= \frac{1+o(1)}{\varepsilon n^2},
\end{equation}
where the last inequality follows from the relations $\theta\le 1/n$ and $1-\theta n=\varepsilon(1+o(1))$.

Next, we turn to bound the probability of $B$. Let $C_x$, for $x\in\{u,v\}$, denote the component of the vertex $x$ in the graph $\bbI-e$. We further denote $\mathbb K_1:=\bbGn\setminus(C_u\cup C_v)$ and $\mathbb K_2:=\bbGne\setminus(C_u\cup C_v)$.

\begin{claim}
\label{claim: external cycles}
For every  $C_u,C_v, \mathbb K_1,\mathbb K_2$ as above there holds
\[
\bbP(B~|~C_u,C_v,\mathbb K_1,\mathbb K_2) \le \mathbf{1}_{C_u\ne C_v}\cdot \theta \cdot (|C_u||C_v|)^2\cdot\prod_{i=1}^{2}\left(\rho+\rho^2\sum_{j\ge 1}|C_j(\mathbb K_i)|^2\right)\,,
\]
where $
\rho := p\varepsilon(1-p)/(1-\theta).
$
\end{claim}

\begin{proof}
We first note that $C_u$ is either equal or disjoint to $C_v$, and that in the former case there exists a path from $v$ to $u$ in $\bbI - e$. We observe that if $C_u=C_v$ then the event $B$ does not occur, hence both sides in the claimed inequality are equal to $0$. Indeed, this is derived directly by combining the facts
$B\subseteq\{e\in\bbI\}\cap A^c$ and 
$A=\{e\in\bbI\}\cap\{C_u=C_v\}.$

Suppose that $C_u\cap C_v=\emptyset$, and consider the edge sets 
$$F_0:=\{\{a,b\}:a\in C_u,~b\in C_v\}\setminus\{e\},$$ and
$$F_1:=\{\{a,b\}:a\in C_u\cup C_v, b\notin C_u\cup C_v\}.$$ Note that for every $f\in F_0\cup F_1$, the only information that is exposed by our conditioning is that $f\notin\bbI$. Therefore, for every two edge subsets $L_1,L_2\subset F_0\cup F_1$ there holds
\begin{equation}\label{eq:indep}
\bbP(L_1\subseteq \bbGn,L_2\subseteq \bbGne~|~C_u,C_v) \le \rho^{|L_1|+|L_2|}
\end{equation}
Indeed, if $L_1\cap L_2\ne \emptyset$ then this conditional probability is $0$ since no edge of $F_0\cup F_1$ is in $\bbI$. Otherwise, by the independence between the different edges, \eqref{eq:indep} follows from the fact that, for every edge $f$,
$
\bbP(f\in\bbGn~|~f\notin\bbI)=\bbP(f\in\bbGne~|~f\notin\bbI)=\rho.
$

We consider two different partitions of $F_1$ given by
\[
F_1 = \bigcup_{j\ge 1,x\in\{u,v\}} F_{x,j,i},~~~~i=1,2\,,
\]
where $F_{x,j,i}$ consists of all the edges between $C_x$ and the $j$-th largest connected component $C_j(\bbK_i)$ of the graph $\bbK_i$. A path from $v$ to $u$ in $\bbGn - e$ can either be {\em internal} and involve an edge from $F_0$, or be {\em external} and involve one edge from $F_{v,j,1}$ and one from $F_{u,j,1}$ for some $j\ge 1$, using the edges from $C_j(\bbK_1)$ to complete the path. Clearly, a similar statement holds for $\bbGne$ where $F_{x,j,1}$ is replaced by $F_{x,j,2}$ for both $x\in\{u,v\}$ (See Figure \ref{fig: exposure of Gn and Gne}). Therefore, we claim that
\begin{align}
 \nonumber
& \bbP(B~|~C_v,C_u,\bbK_1,\bbK_2) \le \\
 \nonumber
&\le~ 
\mathbf{1}_{C_u\ne C_v}\cdot\theta\cdot \left(\rho^2|F_0|^2 +\rho^3|F_0|\sum_{i=1}^{2}\sum_{j\ge 1}\ |F_{v,j,i}||F_{u,j,i}| 
+\rho^4\prod_{i=1}^{2}\sum_{j\ge 1}\ |F_{u,j,i}||F_{v,j,i}|\right) \\
\label{eq:Fs}
&=~ 
\mathbf{1}_{C_u\ne C_v}\cdot\theta \cdot \prod_{i=1}^{2}\left(\rho|F_0|+\rho^2\sum_{j\ge 1}\ |F_{u,j,i}||F_{v,j,i}|\right). 
\end{align}
Indeed, every term in the second line corresponds to a different combination of internal and external paths. The first term corresponds to having two internal paths so we have $|F_0|^2$ choices for having an edge from $F_0$ in both graphs, and the probability that the two edges actually appear is at most $\rho^2$ by \eqref{eq:indep}. Similarly, the second term accounts for having one internal and one external path, where for the external path, say in $\bbGn$, we need to choose the component $C_j(\bbK_1)$ we use, as well as an edge from $F_{u,j,1}$ and an edge from $F_{v,j,1}$. We multiply by $\rho^3\cdot|F_0|$, since in addition to having these two edges appear in $\bbGn$, we also choose an edge from $F_0$ to appear in $\bbGne$. 
The last term is derived by considering the case of two external paths, as we need to choose, for both graphs $
\bbK_i$, a component $C_j(
\bbK_i)$, and edges from $F_{v,j,i}$ and $F_{u,j,i}.$
To conclude, note the multiplicative term $\theta$ accounting for the event $e\in\bbI$.
Alternatively, \eqref{eq:Fs} can be understood as letting each of the graphs $\bbGn,\bbGne$ either choose an internal path with a cost of $\rho$ or an external path with a cost of $\rho^2$. The product of these two terms appears due to the negative correlations from \eqref{eq:indep}. 
The claim is derived from \eqref{eq:Fs} by noting that $|F_0|<|C_u||C_v|$, $|F_{x,j,i}|=|C_x||C_j(\bbK_i)|$ for every $x,j$ and $i$, and a straightforward manipulation.
\end{proof}

We proceed by observing that
\begin{equation}
\label{eq:mon}
\sum_{j\ge 1}|C_j(\mathbb K_1)|^2\le 
\sum_{j\ge 1}|C_j(\bbGn)|^2~\mbox{ and }~ \sum_{j\ge 1}|C_j(\mathbb K_2)|^2 \le \sum_{j\ge 1}|C_j(\bbGne)|^2\,,
\end{equation}
since $\mathbb K_1,\mathbb K_2$ are subgraphs of $\bbGn,\bbGne$ respectively.

Next, for a positive $c\in \Real$, denote by $E_c$  the event that 
\begin{equation*}
\max\left\{\sum_{j\ge 1}|C_j(\bbGn)|^2,\sum_{j\ge 1}|C_j(\bbGne)|^2\right\} \le cn^{4/3},   %\label{eq:Ec_def}
\end{equation*}
and recall that ~\cite[Theorem 1]{Lubetzky2020},\cite{Frilet2021} establish that if $\varepsilon^3n\to\infty$ then the pair
$$
n^{-2/3}\cdot\left( (|C_j(\bbGn)|)_{j\ge 1}, (|C_j(\bbGne)|)_{j\ge 1} \right)
$$ weakly converges in $\ell_2$ to a pair of independent copies of a random sequence whose law was identified by Aldous~\cite{Aldous1995}. Therefore,
\begin{equation}\label{eq:limEc}
\lim_{c\to\infty}\lim_{n\to\infty}\bbP(E_c)=1.    
\end{equation}

By combining  \eqref{eq:mon} and Claim \ref{claim: external cycles}  we find that
\begin{equation}
    \bbP\left( B,E_c ~|~C_v,C_u\right) \leq \mathbf{1}_{C_u\ne C_v}\cdot \theta \cdot (|C_u||C_v|)^2\cdot\left(\rho+c\cdot\rho^2 n^{4/3}\right)^2\,.
\label{eq:P_with_Cs}    
\end{equation}

Let $Y$ denote the size of the connected component of a fixed vertex in a $\calG(n,\theta)$ random graph. Note that for every choice of $C_u$, the random variable $\mathbf{1}_{C_u\ne C_v}|C_v|^2$ is stochastically bounded from above by $Y^2$. Indeed, If $v\in C_u$ then $\mathbf{1}_{C_u\ne C_v}=0$. Otherwise, $C_v$ is the component of $v$ in the $\mathcal G(n-|C_u|,\theta)$ random graph $\bbI\setminus C_u$. As a result, $|C_v|$ is indeed dominated by $Y$. Therefore,
\begin{align}
\nonumber
\bbE[\mathbf{1}_{C_u\ne C_v}(|C_u||C_v|)^2] =~ &\bbE\left[|C_u|^2\cdot  \bbE\left[ \left. \mathbf{1}_{C_u\ne C_v}|C_v|^2~\right|~C_u\right]\right]\\ 
\le~& 
\bbE\left[|C_u|^2 \right]\bbE[Y^2] \nonumber \\ \le~& \bbE[Y^2]^2. \label{eq:Y2a}
\end{align}

In addition,
\begin{equation}\label{eq:Y2b}
\bbE[Y^2] = \frac 1n\bbE_{G\sim\mathcal G(n,\theta)}\left[ \sum_{j\ge 1} |C_j(G)|^{3} \right]\le \frac{1}{(1-n\theta)^{3}} = \frac{1+o(1)}{\varepsilon^3},
\end{equation}
where the first equality is derived by averaging over the vertices and accounting for the contribution of each connected component, the inequality follows from the work of Janson and {\L}uczak on subcritical random graphs~\cite{Janson2008}, and the second equality by $1-n\theta = (1-o(1))\varepsilon$.
By assigning \eqref{eq:Y2a}, \eqref{eq:Y2b}, and the relations $\theta<1/n$ and $\rho = (1+o(1))\varepsilon/n$ in \eqref{eq:P_with_Cs}, we find that
\begin{equation}
    \bbP(B,E_c) \le \frac{1+o(1)}{\varepsilon^6 n}\cdot\left(\frac{\varepsilon}{n} + \frac{c\varepsilon^2}{n^{2/3}} \right)^2 =
\frac{1+o(1)}{\varepsilon n^2}\left( (\varepsilon^3 n)^{-1/2} + c(\varepsilon^3 n)^{-1/6}  \right)^2.
\label{eq:BEc}
\end{equation}

Therefore, we derive from \eqref{eq:A}, \eqref{eq:BEc} and $\varepsilon^3n\to\infty$ that
\[
\bbE[|\bbJ|\cdot \mathbf{1}_{E_c}] = 
\binom n2\bbP(e\in\bbJ,E_c)\le \frac{n^2}{2}(\bbP(A) + \bbP(B,E_c))\le  \frac{1+o(1)}{2\varepsilon}.
\]
Finally, note that by Markov's inequality,
\begin{align*}
\bbP(|\bbJ|>\omega\varepsilon^{-1}) \le&~ \bbP(E_c^c) + \bbP(|\bbJ|\cdot \mathbf{1}_{E_c}>\omega\varepsilon^{-1})\\
\le &~1-\bbP(E_c)+\frac{1+o(1)}{2\omega}.
\end{align*}
This concludes the proof using \eqref{eq:limEc} and the assumption that $\omega\to\infty$ as $n\to\infty$.
\end{proof}

We now apply  Lemma \ref{lem:com_core} to show that the $\cKj(\bbGn,\bbGne,\cbI)$-distributed pair $(\bbMn,\bbMne)$ is close to 
$(\bbFn,\bbFne)$, a pair of random forests that, conditioned on $\bbGn,\bbGne$, is $\cK(\bbGn)\times\cK(\bbGne)$-distributed. In other words, to sample $(\bbFn,\bbFne)$, we first sample the pair $(\bbGn,\bbGne)$ and then apply two independent runs of the cycle-breaking algorithm. We stress that, unconditionally, $\bbFn$ and $\bbFne$ are not independent, due to the dependence between $\bbGn$ and $\bbGne$. To state this claim accurately, we consider the scaled versions $F_{n,\lambda}:=\cS(\bbFn)$ and $F_{n,\lambda}^\varepsilon:=\cS(\bbFn)$. 

\begin{lemma}\label{lem:coupling}
    Fix $\lambda\in\Real$ and let $\varepsilon^3 n\to\infty$. There exists a coupling  of $(\bbMn,\bbMne)$ and $(\bbFn,\bbFne)$ such that $\bbMn=\bbFn$ and
    \begin{equation}
         \label{eq: lemma coupling eq}
             d_{\GHP}^4(M_{n,\lambda}^\varepsilon,F_{n,\lambda}^\varepsilon)\arp 0\,,
         \end{equation}
         as $n\to\infty$.
\end{lemma}

\begin{proof}
Recall that $\bbJ =\Conn(\bbGn)\cap \Conn(\bbGne)$, and $\cbI=\{e\in K_n~:~w_e\le p,b_e=0\}$ is the random graph consists of the edges in $\bbGn\cap\bbGne$ whose weight had not been resampled. We sample the graphs $\bbGn,\bbGne,\bbMn,\bbMne$ using $W_n,W_n^\varepsilon$ (See \S\ref{Section: Preliminaries}), and set $\bbFn:=\bbMn$.
In addition, let $\bbFne$ be the MSF of $\bbGne$ endowed with the following edge weights: 
\[
\tilde w_e = 
\begin{cases}
    w_e^\varepsilon & e\in\bbGne\setminus(\cbI\cap\bbJ), \\
    p\cdot w_e' & e\in\cbI\cap\bbJ,
\end{cases}
\]
where $w_e'$ is an independent $\mathrm{U}[0,1]$ variable.
First, we claim that the forests $\bbFn,\bbFne$ are retained respectively from $\bbGn,\bbGne$ by independent cycle breaking algorithms. Namely, conditioned on $\bbGn,\bbGne$,the pair $(\bbFn,\bbFne)$ is $\cK(\bbGn)\times\cK(\bbGne)$-distributed. This follows from the fact that conditioned on $\bbGn,\bbGne$ and $\cbI$, the weights 
$$ (w_e)_{e\in\Conn(\bbGn)} \mbox{ and }(\tilde w_e)_{e\in\Conn(\bbGne)},$$
which determine the edges that are removed in the cycle breaking algorithms, are i.i.d. Indeed, the only dependency between weights can occur via an edge from $\bbJ$ but for every such an edge $e$, the weights in both graphs are independent either due to resampling (if $e\notin\cbI$) or by the definition of $\tilde w_e$ (if $e\in\cbI$).

Next, we bound the distance  $d_{\GHP}^4(M_{n,\lambda}^\varepsilon,F_{n,\lambda}^\varepsilon).$
Denote by $B_j,~j\ge 1$, the event that the trees $C_j(\bbMne)$ and $C_j(\bbFne)$ are different. Note that the forests $\bbMne$ and $\bbFne$ are retained from $\bbGne$ by the cycle-breaking algorithm using, respectively, the edge weights $(w_e^\varepsilon)_{e\in\bbGne}$ and $(\tilde w_e)_{e\in\bbGne}$, which differ only on $\cbI\cap\bbJ$. 
Therefore, if $B_j$ occurs then there exists a cycle $\gamma$ in $C_j(\bbGne)$ and an edge $f\in \gamma\cap\cbI\cap\bbJ$ that is the heaviest in $\gamma$ with respect to one of the edge weights. Otherwise, the two runs of the cycle-breaking algorithms on $C_j(\bbGne)$  must be identical. 

Let $S$ denote the number of distinct simple cycles in $C_j(\bbGne)$, $R$ the length of the shortest cycle in $C_j(\bbGne)$ (or $R=\infty$ if the component is acyclic), and let $\gamma$ be a cycle in $C_j(\bbGne)$. Conditioned on $\bbGne$ and $\bbJ$, the probability that the heaviest edge of $\gamma$ (in each of the weights) belongs to  $\bbJ$ is bounded from above by $|\bbJ|/{R}$, since  $|\gamma|$ is bounded from below by $R$. Hence, by taking the union bound over all the cycles in the component and the two edge weights we find that $\bbP(B_j~|~\bbGne,\bbJ) \leq 2 \cdot S \cdot |\bbJ|/R.$
Therefore, for every $\omega > 0$, the probability of $B_j$ conditioned on  the event $C$ that $|\bbJ|< \omega  \varepsilon^{-1},~S< \omega,$ and $R > n^{1/3}\omega^{-1}$, is bounded by
$$
\bbP\left(B_j~|~C\right) \leq  \frac{2\cdot \omega\cdot (\omega\varepsilon^{-1}) }{n^{1/3}\omega^{-1}}.
$$
Consequently,
\begin{equation}
\bbP(B_j) \le \bbP(|\bbJ|\ge \omega \varepsilon^{-1}) + 
\bbP(S\ge \omega) + 
\bbP(R\le n^{1/3}\omega^{-1}) +
\frac{2\cdot \omega\cdot (\omega\varepsilon^{-1}) }{n^{1/3}\omega^{-1}}.\label{eq:Bj} 
\end{equation}

Suppose that $\omega=\omega(n)\to\infty$ as $n\to\infty$.
Lemma \ref{lem:com_core} asserts that first term in \eqref{eq:Bj} is negligible. In addition, the second and third terms are also negligible by known results on critical random graphs. Namely, $S$ converges in distribution to an almost-surely finite limit by~\cite{Luczak1994,Aldous1995}, and the fact that $n^{-1/3}\omega R$ almost surely diverges follows from \cite{Addario-Berry2010} for unicyclic components, and from \cite{Luczak1994} for complex components (components with more than one cycle). Choosing $\omega=\omega(n)$ such that $\omega\to \infty$ and $\omega^3 / (\varepsilon n^{1/3})\to 0$ as $n\to\infty$ results in $\bbP(B_j)\to 0$, for every $j\geq 1$.

To complete the proof,  observe that for every $\eta>0$ and $N\ge 1$ there holds
\begin{equation}
    \bbP(d_{\GHP}^4(M_{n,\lambda}^\varepsilon,
F_{n,\lambda}^\varepsilon)>\eta) \le
\sum_{j=1}^{N-1} \bbP(B_j) + \bbP\left( 
\sum_{j=N}^\infty d_{\GHP}(M_{n,\lambda,j}^\varepsilon,
F_{n,\lambda,j}^\varepsilon)^4 > \eta
\right).
%\label{eq:easy}
\nonumber 
\end{equation}
The first sum is negligible as $n\to\infty$ since $\bbP(B_j)\to 0$ for every $j\ge 1$. In addition, by the fact that both $M_{n,\lambda}^\varepsilon$ and $F_{n,\lambda}^\varepsilon$ converge in distribution as $n\to\infty$ in $(\mathbb L_4,d_{\GHP}^4)$ we have that
\[
\lim_{N\to\infty}\limsup_{n\to\infty} 
\bbP\left( 
\sum_{j=N}^\infty d_{\GHP}(M_{n,\lambda,j}^\varepsilon,
F_{n,\lambda,j}^\varepsilon)^4 > \eta
\right)= 0,   
\]
which completes the proof of the lemma.
\end{proof}

Next, we turn to derive the asymptotic independence of the rescaled measured metric spaces $F_{n,\lambda}$ and $F_{n,\lambda}^\varepsilon$. 

 \begin{lemma}\label{lem:indKs}
  Fix $\lambda\in\Real$ and suppose that $\varepsilon^3n\to\infty$ as $n\to\infty.$ Then, the pair $(F_{n,\lambda},F_{n,\lambda}^\varepsilon)$ converges in distribution to a pair of independent copies of $\scrM_{\lambda}$ in $(\mathbb{L}_4,d_{\GHP}^4)$ as $n\to\infty$.
 \end{lemma}
\begin{proof}
% We need to show that for every continuity sets $S,S'$ of $(\mathbb{L}_4,d_{\GHP}^4)$ for the distribution of $\scrM_\lambda$ there holds
% \[
% \bbP(F_{n,\lambda}\in S,F^\varepsilon_{n,\lambda}\in S') \to 
% \bbP(\scrM_\lambda\in S)\bbP(\scrM_\lambda\in S')\,,\mbox{as } n\to\infty.
% \]
% Since 
We start by describing, in very high-level terms, how the the space $\scrM_\lambda$ is constructed. Recall the random measured metric space $\scrG_\lambda$ that was introduced in ~\cite{AddarioBerry2009},~\cite{AddarioBerry2013}, and was shown to be the limit of $G_{n,\lambda}$ in distribution, as $n\to\infty$, in $(\mathbb{L}_4,d_{\GHP}^4).$ The random space $\scrM_\lambda$ was defined conditioned on $\scrG_\lambda$ as being $\cK(\scrG_\lambda)$-distributed, where $\cK$ is the continuous analog of the cycle-breaking algorithm.

Next, denote by $(s(G_{n,\lambda,i}))_{i\geq 1}$ and $(s(\scrG_{\lambda,i}))_{i\geq 1}$ the sequence of surpluses of the components in $G_{n,\lambda}$ and $\scrG_{\lambda}$ respectively. In addition, let $(r(G_{n,\lambda,i}))_{i\geq 1}$ and $(r(\scrG_{\lambda,i}))_{i\geq 1}$ be the sequence of minimal length of a core edge in each component. We refer the reader to \cite{AddarioBerry2013} for precise definitions. The following claim follows from the proof of \cite[Theorem 4.4]{AddarioBerry2013}.

\begin{claim}\label{claim: strong}
    Let $\Omega$ be a probability space in which $\bbGn,\scrG_{\lambda}$ are commonly defined such that $\Omega$-almost-surely there holds that
    \begin{align*}
    G_{n,\lambda}&\to\scrG_{\lambda}\mbox{ in }(\mathbb{L}_4,d_{\GHP}^4),\\
    (s(G_{n,\lambda,i}))_{i\geq 1}&\to (s(\scrG_{\lambda,i}))_{i\geq 1},\\
    (r(G_{n,\lambda,i}))_{i\geq 1}&\to (r(\scrG_{\lambda,i}))_{i\geq 1},
    \end{align*}
    as $n\to\infty$.
Then, for every continuity set $S$ of $(\mathbb{L}_4,d_{\GHP}^4)$ for $\scrM_\lambda$, the convergence     
\[
\bbP(M_{n,\lambda} \in S ~|~\bbGn) \to \bbP(\scrM_{\lambda} \in S ~|~\scrG_\lambda)\,,~\mbox{as $n\to\infty$},
\]
of random variables occurs $\Omega$-almost surely. Here, conditioned on $\bbGn$ and $\scrG_{\lambda}$, $\bbMn$ is $\cK(\bbGn)$-distributed, $M_{n,\lambda}:=\cS(\bbMn)$ and $\scrM_\lambda$ is $\cK(\scrG_\lambda)$-distributed.
\end{claim}
In fact, it is proved in \cite[Theorem 4.4]{AddarioBerry2013} that under the conditions of Claim \ref{claim: strong}, the cycle-breaking algorithms carried out on $\bbGn,\scrG_\lambda$ can be coupled such that the convergence of $M_{n,\lambda}\to\scrM_{\lambda}$ in $(\mathbb{L}_4,d_{\GHP}^4)$ also occurs $\Omega$-almost-surely.

Back to noise-sensitivity, the results in ~\cite[Theorem 2]{Lubetzky2020} and \cite[Theorem~9.1.1]{Frilet2021}, establish that if $\varepsilon^3n\to\infty$ as $n\to\infty$ then
\begin{align}
\label{eq: firstconv1}
(G_{n,\lambda},G_{n,\lambda}^\varepsilon) & \ard 
(\scrG_\lambda,\scrG_\lambda'),~\mbox{and} \\
\label{eq: firstconv2}
\left((s(G_{n,\lambda,i}))_{i\geq 1},(s(G^\varepsilon_{n,\lambda,i}))_{i\geq 1}\right)& \ard\left((s(\scrG_{\lambda,i}))_{i\geq 1},(s(\scrG_{\lambda,i}'))_{i\geq 1}\right),
\end{align}
where $\scrG_\lambda'$ is an independent copy of $\scrG_\lambda$. Here the first convergence is in $(\mathbb L_4,d_\GHP^4)$, and the second in the sense of finite dimensional distributions,

The proof of \cite[Theorem 4.1]{AddarioBerry2013} shows that this convergence can be extended to the minimal lengths of core edges, implying that
\begin{align}
\label{eq: firstconv3}
\left((r(G_{n,\lambda,i}))_{i\geq 1},(r(G^\varepsilon_{n,\lambda,i}))_{i\geq 1}\right)& \ard\left((r(\scrG_{\lambda,i}))_{i\geq 1},(r(\scrG_{\lambda,i}'))_{i\geq 1}\right),
\end{align}
as $n\to\infty$.

Using Skorohod's representation theorem, we may work in a probability space $\Omega$ in which the convergences. \eqref{eq: firstconv1},\eqref{eq: firstconv2} and \eqref{eq: firstconv3} occur almost surely. In addition, we can consider the distributions of $F_{n,\lambda},F_{n,\lambda}^\varepsilon,\scrM_\lambda$ and its independent copy $\scrM_\lambda'$ by constructing them via $\Omega$. Namely, conditioned on $\bbGn,\bbGne,\scrG_\lambda,\scrG_\lambda'$ sampled in $\Omega$, we consider the (distributions of the) following random elements:
\begin{itemize}
    \item The pair $(\bbFn,\bbFne)$ is $\cK(\bbGn)\times\cK(\bbGne)$-distributed, $F_{n,\lambda}:=\cS(\bbFn)$ and $F_{n,\lambda}^\varepsilon = \cS(\bbFne)$, and
    \item The pair $(\scrM_\lambda,\scrM_\lambda')$ is $\cK(\scrG_\lambda) \times \cK(\scrG_\lambda')$-distributed.
\end{itemize}
We observe that the pair $(F_{n,\lambda},F_{n,\lambda}^\varepsilon)$ is conditionally independent given $\bbGn,\bbGne$, and the pair $(\scrM_\lambda,\scrM_\lambda')$ is independent and identically distributed. 
%Note that since we are interested in convergence in distribution, we are oblivious to whether or not the cycle-breaking procedures of $(\scrM_\lambda,\scrM_\lambda')$ are correlated with those of $(F_{n,\lambda},F_{n,\lambda}^\varepsilon)$.

Our goal is to show that for every continuity sets $S,S'$ of $(\mathbb L_4,d_{\GHP}^4)$ for the distribution of $\scrM_\lambda$ there holds
\[
\bbP(F_{n,\lambda}\in S,F^\varepsilon_{n,\lambda}\in S') \to 
\bbP(\scrM_\lambda\in S)\bbP(\scrM_\lambda'\in S')
\]
as $n\to\infty$. 

Note that by our assumption on the almost-sure convergences in $\Omega$, we can apply Claim \ref{claim: strong} twice and obtain that for every two such continuity sets $S,S'$ there holds that both convergences 
\begin{equation}
    \label{eq:event1}
\bbP(F_{n,\lambda}\in S ~|~\bbGn) \to \bbP(\scrM_\lambda\in S ~|~\scrG_\lambda)\,,~\mbox{as $n\to\infty$},
\end{equation}
and
\begin{equation}
    \label{eq:event2}
\bbP(F_{n,\lambda}^\varepsilon\in S' ~|~\bbGne) \to
\bbP(\scrM_\lambda'\in S' ~|~\scrG_\lambda')\,,~\mbox{as $n\to\infty$}
\end{equation}
occur $\Omega$-almost surely.
Consequently, the proof is concluded as follows:
\begin{align*}
\bbP(F_{n,\lambda}\in S,F^\varepsilon_{n,\lambda}\in S')&=~
\bbE[\bbP(F_{n,\lambda}\in S,F^\varepsilon_{n,\lambda}\in S'~|~\bbGn,\bbGne)]\\
&=~
\bbE[\bbP(F_{n,\lambda}\in S ~|~\bbGn)\cdot\bbP(F_{n,\lambda}^\varepsilon\in S' ~|~\bbGne)]\\
&\to~
 \bbE[\bbP(\scrM_\lambda\in S ~|~\scrG_\lambda)\cdot\bbP(\scrM_\lambda'\in S' ~|~\scrG_\lambda')]\\
 &=~
 \bbE[\bbP(\scrM_\lambda\in S ~|~\scrG_\lambda)]\cdot \bbE[\bbP(\scrM_\lambda'\in S' ~|~\scrG_\lambda')]\\
&=~\bbP(\scrM_\lambda\in S)\cdot \bbP(\scrM_\lambda'\in S').
\end{align*}
The first equality holds by the law of total expectation, and the second equality is due to the conditional independence of $F_{n,\lambda},F_{n,\lambda}^\varepsilon$ given $\bbGn,\bbGne$. The convergence, which occurs as $n\to\infty$, is obtained from the $\Omega$-almost-sure convergences \eqref{eq:event1}, \eqref{eq:event2} and using the dominated convergence theorem. The next equality is obtained by the independence of $\scrG_\lambda$,$\scrG_\lambda'$ which is where the noise-sensitivity of the measured metric structure of $\bbGn,\bbGne$ is being used. The last equality follows from the law of total expectation. 

\end{proof}

We conclude this subsection with a proof of the noise-sensitivity of the MSF of $\bbGn$, which we derive from the following well-known theorem.
\begin{theorem}[{\cite[Theorem~3.1]{billing1999}}]
    \label{thm: slutzky}
    Let $S$ be a Polish space with metric $\rho$ and $(X_n,Y_n)$ be random elements of $S\times S$. If $~Y_n\ard X$ and $\rho(X_n,Y_n)\arp 0$ as $n\to\infty$, then $X_n\ard X$.
    \end{theorem}
\begin{proof}[Proof of Theorem \ref{theorem: secondtheorem}, Part \eqref{part: GNP1}]
Denote the polish metric space $S=(\mathbb L_4,d_{\GHP}^4)^2$ endowed with some product metric $\rho$. Suppose that the random elements $$((M_{n,\lambda},M_{n,\lambda}^\varepsilon),(F_{n,\lambda},F_{n,\lambda}^\varepsilon)) \in S\times S$$ are sampled via the coupling from Lemma \ref{lem:coupling}.
 Lemma \ref{lem:indKs} asserts that $(F_{n,\lambda},F_{n,\lambda}^\varepsilon)$ converges in distribution to a pair of independent copies of $\scrM_{\lambda}$. In addition, by Lemma \ref{lem:coupling},
\[
\rho((M_{n,\lambda},M_{n,\lambda}^\varepsilon),(F_{n,\lambda},F_{n,\lambda}^\varepsilon))\arp 0\,,\]
as $n\to\infty$.
Consequently, we derive from Theorem \ref{thm: slutzky} that $(M_{n,\lambda},M_{n,\lambda}^\varepsilon)$ converges in distribution to a pair of independent copies of $\scrM_{\lambda}$, as claimed.
\end{proof}

\subsection{Noise Stability of the MSF}
\label{Subsection: Noise Stability}
We now assume that $\varepsilon^3n\to 0$ as $n\to \infty$. In this case, the noise stability of the MSF follows from the similarity between the cycle breaking algorithms. Namely, the  $\cKj(\bbGn,\bbGne,\cbI)$-distributed pair $(\bbMn,\bbMne)$ is obtained by removing the exact same set of edges from both graphs.
We derive this from the following claim which asserts that all the cycles in $\bbGn$ and $\bbGne$ appear in their common subgraph $\cbI$ consisting of the edges whose weight was not resampled.
\begin{claim}
    \label{claim: core equality}
    Let $\lambda \in \Real$, $j\geq 1$, $\varepsilon^3 n \to 0 $, and $\bbGn$ and $\cbI$ defined  as in \S\ref{subsec: notations}. Let $B_j$ denote the event that 
    $\Conn(C_j(\bbGn)) =\Conn(C_j(\cbI))$. Then, $\bbP(B_j)\to 1$ as $n\to \infty$. %In consequence,
% \begin{equation}
% \label{eq: claim core equality eq cor}
%      B_j:=\{\}.
% \end{equation}

\end{claim}
\begin{proof}
We observe that conditioned on $\bbGn$, the graph $\cbI$ is obtained from $\bbGn$ by removing each edge independently with probability $\varepsilon$. Therefore, by \cite[Lemma~5.4]{Lubetzky2020}, the event $A_j$ that $C_j(\cbI)\subseteq C_j(\bbGn)$ occurs with probability tending to $1$ as $n\to\infty.$
In addition, under the event $A_j$, the event $B_j$ does not occur only if there exists an edge $e\in\Conn(C_j(\bbGn))$ that $\cbI$ did not retain. Therefore, for every $\omega=\omega(n)>0$,
\begin{equation}\label{eq:Bjc}
\bbP(B_j^c~|~A_j) \le \bbP(|\Conn(C_j(\bbGn))|>\omega n^{1/3}) + \varepsilon\omega n^{1/3},    
\end{equation}

 where the second term bounds the expected number of edges from $\Conn(\bbGn)$ that $\cbI$ did not retain, conditioned on $|\Conn(C_j(\bbGn))|\le \omega n^{1/3}$. We derive the claim by combining $\bbP(A_j)\to 1$ and assigning in \eqref{eq:Bjc} a sequence $\omega = \omega(n)$ such that $\omega\to\infty$ and $\varepsilon\omega n^{1/3}\to 0$ as $n\to\infty$. Indeed, in such a case we have that $\bbP(n^{-1/3}|\Conn(C_j(\bbGn))|>\omega)\to 0$, since the maximum number of cycles in $\bbGn$ is bounded in probability~\cite{Luczak1994}, and so is the length of the largest cycle in $C_j(\bbGn)$ divided by $n^{1/3}$~\cite{Luczak1994,Addario-Berry2010}. 
 \end{proof}

Note that the assertion in Claim \ref{claim: core equality} also holds for $\bbGne$ since $(\bbGn,\cbI)\stackrel{d}{=}(\bbGne,\cbI)$. We now turn to conclude the proof of Theorem \ref{theorem: secondtheorem}.

\begin{proof}[Proof of  Theorem \ref{theorem: secondtheorem} Part \ref{part: GNP2}]
Denote by $\cbM$ the MSF of the graph $\cbI$ endowed with the edge weights from $W_n$, and let $\check M:=\cS(\cbM)$. First, we argue that  for every fixed $j \ge 1$, 
\begin{equation}\label{eq:convergence_j}
d_{\GHP}(M_{n,\lambda,j},\check M_j) \arp 0 \end{equation}
as $n\to \infty$.
By Claim \ref{claim: core equality} we can condition on the event $B_j$. Under this event, the joint cycle breaking algorithm running on $C_j(\bbGn)$ and $C_j(\cbI)$ removes the same edges in both graphs. Since $\cbI$ is a subgraph of $\bbGn$, we deduce that 
$C_j(\bbMn)$ is obtained from $C_j(\cbM)$ by the addition of the forest $C_j(\bbGn) \setminus C_j(\cbI)$. 
We derive \eqref{eq:convergence_j} by the proof of \cite[Theorem 2]{Lubetzky2020}, which shows that with probability tending to $1$ as $n\to\infty$, the graph $C_j(\bbGn)$ is contained in a neighborhood of radius $o(n^{1/3})$ around $C_j(\cbI)$, and that the two graphs differ by $o(n^{2/3})$ vertices.

Since $(\bbGn,\cbI,W_n)\stackrel{d}{=}(\bbGne,\cbI,W_n^\varepsilon)$, we can use the same argument for $\bbGne$ instead of $\bbGn$, and find that
\[
d_{\GHP}(M_{n,
\lambda,j},M_{n,
\lambda,j}^\varepsilon) \arp 0,
\]
as $n\to\infty$. To conclude Theorem \ref{theorem: secondtheorem} Part \ref{part: GNP2} we need to extend the component-wise convergence to $\left(\bbL_4,d_{GHP}^4\right)$. This is carried out exactly as in the proof of Lemma \ref{lem:coupling}, following \cite[Theorems 4.1, 4.4]{AddarioBerry2013}.
\end{proof}

\section{Proof of Theorem \ref{theorem: maintheorem}}
\label{sec:MST}
The connection between the scaling limits of the MST $\bbM_n$ and the largest component of the MSF $\bbMn$ was established in \cite[Proposition~4.8]{AddarioBerry2013}. Let  $M_n = \hat{\cS}(\bbMn)$ and $\hat{M}_{n,\lambda,1} = \hat{\cS}(\bbM_{n,\lambda,1})$ (see \S\ref{subsec: notations}).
Then, for every $\eta>0$, 
\begin{equation}\label{eq:4.8}
    \lim_{\lambda \to \infty}\limsup_{n\to\infty}\bbP\left(d_{\GHP}(M_n,\hat{M}_{n,\lambda,1})>\eta\right)=0,
\end{equation}
and a similar statement holds for $M_n^\varepsilon,\hat{M}_{n,\lambda,1}^\varepsilon$. 

In addition, let $\hat{\scrM}_{\lambda,1}$ be the measured metric space obtained from the scaling limit $\scrM_{\lambda,1}$ by renormalizing its measure to a probability measure. The so-called principle of accompanying laws ~\cite[Theorem 9.1.13]{Stroock} yields that $M_n \ard \scrM$ in $(\calM,d_{\GHP})$, where the random measured metric space $\scrM$ is the limit of  $\hat{\scrM}_{\lambda,1} \ard \scrM$ in $d_{\GHP}$ as $n\to\infty$. Given this background, Theorem \ref{theorem: maintheorem} follows directly from Theorem \ref{theorem: secondtheorem}.

\begin{proof}[Proof of Theorem \ref{theorem: maintheorem}]
For Part \ref{part: 1}, we let $\rho$ be some product metric on $(\calM,d_{\GHP})^2$, and deduce from \eqref{eq:4.8} that for every $\eta>0,$
\[
\lim_{\lambda \to \infty}\limsup_{n\to\infty}\bbP\left(\rho( (M_n,M_n^\varepsilon) , (\hat{M}_{n,\lambda,1},\hat{M}_{n,\lambda,1}^\varepsilon))>\eta\right)=0.
\]
In addition, Theorem \ref{theorem: maintheorem} Part \ref{part: GNP1} implies that $$\left(\hat{M}_{n,\lambda,1},\hat{M}^\varepsilon_{n,\lambda,1}\right)\ard \left(\hat{\scrM}_{\lambda,1},\hat{\scrM}_{\lambda,1}'\right),$$ 
in $(\calM,d_{\GHP})^2$, as $n\to\infty$. Let  $\hat{\scrM}_{\lambda,1}'$ to be an independent copy of $\hat{\scrM}_{\lambda,1}$, hence
\[
\left(\hat{\scrM}_{\lambda,1},\hat{\scrM}_{\lambda,1}'\right)\ard \left(\scrM,\scrM'\right),
\]
as $\lambda\to\infty$ in $d_{\GHP}$, where $\scrM$ and $\scrM'$ are i.i.d.
Therefore, by the principle of accompanying laws, as $n\to\infty$, the pair $(M_n,M_n^\varepsilon)\ard\left(\scrM,\scrM'\right)$ in $d_{\GHP}$.

For Part 2, note that for every $\eta>0$ and every $\lambda\in\Real$ there holds
\[
\bbP(d_{\GHP}(M_n,M_n^\varepsilon)>\eta) \le
\bbP(D_1) +
\bbP(D_2) +
\bbP(D_3),
\]
where $D_1,D_2$ and $D_3$ are the events that the GHP distance between $ \left(M_n,\hat{M}_{n,\lambda,1}\right)$, $\left(M_n^\varepsilon,\hat{M}_{n,\lambda,1}^\varepsilon\right)$ and $\left(\hat{M}_{n,\lambda,1},\hat{M}^\varepsilon_{n,\lambda,1}\right)$ is greater than $\eta /3$, respectively.

Part \ref{part: GNP2} of Theorem \ref{theorem: secondtheorem} implies that  
$ 
d_{\GHP}\left(\hat{M}_{n,\lambda,1},\hat{M}^\varepsilon_{n,\lambda,1}\right)\arp 0
$
as $n\to\infty$, thereby $\bbP(D_3)\to 0$. By applying \eqref{eq:4.8} to both $ \left(M_n,\hat{M}_{n,\lambda,1}\right)$ and $ \left(M_n^\varepsilon,\hat{M}_{n,\lambda,1}^\varepsilon\right)$, we find that 
\[
\lim_{\lambda\to\infty}\limsup_{n\to\infty}\bbP(D_1)+\bbP(D_2)=0,
\]
therefore 
$\bbP(d_{\GHP}(M_n,M_n^\varepsilon)>\eta) \to 0$ as $n\to\infty$, as claimed.
\end{proof}

\section{Open Problems}
\label{sec:open}
We conclude with three open problems that naturally arise from our work.
First, it will be interesting to study the joint limit law of the scaled MSTs $(M_n,M_n^\varepsilon)$  and of the scaled MSFs $(M_{n,\lambda},M_{n,\lambda}^\varepsilon)$ in the critical noise regime $\varepsilon = tn^{-1/3},~t\in\Real$. Rossignol~\cite{Rossignol2021} identified a non-trivial correlation between $\bbGn$ and $\bbGne$, but we suspect that the correlations between the MSFs are even more involved. Namely, in this regime the subgraphs $\Conn(\bbGn)$ and $\Conn(\bbGne)$ share a positive fraction of their {\em weighted} edges. Hence, on top of the correlations between $\bbGn$ and $\bbGne$, the joint cycle-breaking algorithm retaining $\bbMn,\bbMne$ is also non-trivially correlated.

Second, even though this paper considers $\mathrm{U}[0,1]$-distributed weights, our setting can be equivalently described in discrete terms. It is also natural to consider similar problems in a continuous noise model, e.g., by letting $(w_e,w_e^\varepsilon)$ be identically distributed normal variables with covariance $\varepsilon$. We ask: what is the sensitivity-stability noise threshold of the scaled MST in this model? is it still aligned with the critical window of the Erd\H{o}s-R\'enyi random graphs?

Third, it is interesting to explore for which functions of the MST our theorem establishes noise sensitivity and stability. This requires a better understanding of the limit $\scrM$ and its continuity sets. For example, consider the diameter of $\scrM$ or the distance between two independent random points in it. Are these random variables continuous? What are their support? It is not entirely clear to us how to answer these questions using the construction in~\cite{AddarioBerry2013} of $\scrM$ as the limit of $\scrM_{\lambda}$ as $\lambda\to\infty$. However, it appears that the recent explicit construction of $\scrM$ as the Brownian parabolic tree ~\cite{broutin2023convex} can be quite useful here.

\medskip

\textbf{Acknowledgment.} We thank the anonymous referee for their very helpful comments and suggestions.

\bibliographystyle{abbrv}
\bibliography{MST-bib}
\end{document}